\documentclass[11 pt]{amsart}
\usepackage{amsmath,amsthm,amsfonts,amssymb,amscd,amsrefs}
\usepackage[top=3cm, bottom=3cm,left=2cm, right=2cm]{geometry}
\allowdisplaybreaks[1]
\usepackage{enumitem}
\setlist[enumerate]{font=\normalfont}
\usepackage{centernot}
\usepackage{xcolor}
\usepackage{soul}
\usepackage{tikz}
  \usepackage[linkcolor=red,citecolor=green,colorlinks=true]{hyperref}
\usepackage{mathrsfs}
\usepackage{mathtools}
\usepackage{etoolbox}
\usepackage{scalerel,stackengine}
\stackMath
\newcommand\widecheck[1]{%
\savestack{\tmpbox}{\stretchto{%
  \scaleto{%
    \scalerel*[\widthof{\ensuremath{#1}}]{\kern-.6pt\bigwedge\kern-.6pt}%
    {\rule[-\textheight/2]{1ex}{\textheight}}
  }{\textheight}%
}{0.5ex}}%
\stackon[1pt]{#1}{\scalebox{-1}{\tmpbox}}%
}
\theoremstyle{definition}
\newtheorem{defn}{Definition}[section]
\newtheorem{qn}[defn]{Question}
\newtheorem{ex}[defn]{Example}

\newtheorem{case}{Case}
\theoremstyle{remark}

\newtheorem{rem}[defn]{Remark}
\theoremstyle{plain}
\newtheorem{thm}[defn]{Theorem}
\newtheorem{step}{Step}
\newtheorem{lem}[defn]{Lemma}

\newcounter{cnt1}
\newcounter{cnt2}
\newcounter{cnt3}
\newcounter{cnt4}
\newcommand{\blr}{\begin{list}{$($\roman{cnt1}$)$} {\usecounter{cnt1}
 \setlength{\topsep}{0pt} \setlength{\itemsep}{0pt}}}
\newcommand{\blR}{\begin{list}{\Roman{cnt4}.\ } {\usecounter{cnt4}
 \setlength{\topsep}{0pt} \setlength{\itemsep}{0pt}}}
\newcommand{\bla}{\begin{list}{$(\alph{cnt2})$} {\usecounter{cnt2}
 \setlength{\topsep}{0pt} \setlength{\itemsep}{0pt}}}
\newcommand{\bln}{\begin{list}{$($\arabic{cnt3}$)$} {\usecounter{cnt3}
 \setlength{\topsep}{0pt} \setlength{\itemsep}{0pt}}}
\newcommand{\el}{\end{list}}

\numberwithin{equation}{section}
\makeatletter 

\newcommand{\Rmnum}[1]{\expandafter\@slowromancap\romannumeral #1@}

\makeatother
\makeatletter
\def\imod#1{\allowbreak\mkern10mu({\operator@font mod}\,\,#1)}
\makeatother
\makeatletter 
\patchcmd{\@settitle}{\uppercasenonmath\@title}{}{}{}
\patchcmd{\@setauthors}{\MakeUppercase}{}{}{}
\patchcmd{\section}{\scshape}{}{}{}
\makeatother
\begin{document}
\selectfont
\title[\tiny{Points of continuity in Bochner function spaces $L^{p}(\mu, X)$}]{Weak$^*$-weak points of continuity and very smooth points in Bochner function spaces $L^{p}(\mu, X)$}
\author{Saurabh Dwivedi}
    \address{Department of Mathematics, Shiv Nadar Institution of Eminence, Delhi-NCR, India}
	\email{sd605@snu.edu.in, saurabhdewedi876@gmail.com}
	
	\begin{abstract}
    Suppose that $([0,1],\mathcal{A},\mu)$ is the Lebesgue measure space and $1< p<\infty$ with $\frac{1}{p}+\frac{1}{q}=1$. We prove that an unit vector $h\in L^{p}(\mu, X)^*$ is a weak$^*$-weak point of continuity of $S(L^{p}(\mu, X)^*)$ if and only if $\frac{h(t)}{\|h(t)\|}$ is a weak$^*$-weak point of continuity of $S(X^*)$ for almost every $t\in [0, 1]$ with $h(t)\neq 0$, where $X$ is a Banach space such that $X^*$ is separable. As an application of this result, under the same hypothesis on $X$, we prove that a point $f\in L^p(\mu, X)_{1}$ is a very smooth point  if and only if $\frac{f(t)}{\|f(t)\|}\in S(X)$ is a very smooth point for almost every $t\in [0, 1]$ with $f(t)\neq 0$. We give several examples to show that these results does not hold for $p=1$. 
\end{abstract}
\maketitle

\section{Introduction}
Throughout, $X$ denote a real or complex Banach space, $X^*$ the dual space of $X$, $X_1$ the unit ball of $X$ and $S(X)$ the unit sphere of $X$, $([0,1],\mathcal{A},\mu)$ the Lebesgue measure space, $p$ and $q$ positive numbers with $\frac{1}{p}+\frac{1}{q}=1$. For any set $A$, we denote by $A^c$ the set-theoretic complement of $A$.

The geometric properties of the unit ball of the Bochner function spaces $L^{p}(\mu, X)$ have been studied extensively. Characterizations have been obtained for extreme points, points of continuity, denting points, and strongly extreme points of the unit ball of $L^{p}(\mu, X)$. (see, for example, \cites{HL93, PG1, LP, PG2}). However, the study of weak$^*$-weak points of continuity (and weak$^*$-norm points of continuity) of the dual unit ball of $L^{p}(\mu, X)$ has not been studied in detail. In this paper, we study weak$^*$-weak points of continuity (and weak$^*$-norm points of continuity) of the dual unit ball of $L^{p}(\mu, X)$ and their relation to weak$^*$-weak points of continuity (and weak$^*$-norm points of continuity) of the dual unit ball of $X$. We also provide some applications of these results to study very smooth points of $L^{p}(\mu, X)$.

Let $K\subseteq X_{1}^*$ be a norm closed and bounded subset of $X_{1}^*$. We recall from \cite{HL93} that a point $x^*\in K$ is said to be a weak$^*$-norm point of continuity (\textnormal{w}$^*$- $\|\cdot\|$ PC) if the identity map $I:(K,\textnormal{w}^*)\rightarrow(K,\|.\|)$ is continuous at $x^*$, and a weak$^*$-weak point of continuity (\textnormal{w}$^*$-\textnormal{w} PC) if the identity map $I:(K,\textnormal{w}^*)\rightarrow(K,\textnormal{w})$ is continuous at $x^*$. When $K=S(X^*)$, we refer the reader to \cite{SMR} for a detailed study of w$^*$-w PC's of $K$. Here the authors showed that if $X$ is not reflexive, then every w$^{*}$-w PC of $X_{1}^{*}$ lies in $S(X^{*})$. They also proved that a functional $x^{*}$ with $\|x^{*}\| = 1$ is a w$^{*}$-w PC of $X_{1}^{*}$ if and only if it is a w$^{*}$-w PC of $S(X^{*})$. These points of continuity are closely related to several geometric notions in Banach space theory, such as G\^ateaux and Fr\'echet differentiability of the norm, denting points, extreme points, and the uniqueness of norm-preserving extensions of functionals on $X$ to $X^{**}$ (see \cites{GI, GGS, HWW}).

To provide further motivation for our study, let us recall the definitions of a denting point and of a \textnormal{w}$^*$-denting point. Let $C$ be a bounded subset of $X$. and let $x^*\in X^*$ and $\alpha>0$. Then the set $$S(C, x^*, \alpha)=\{x\in C:x^*(x)>\operatorname{sup}x^*(C)-\alpha\}$$ is called a slice of $C$ determined by $x^*$ and $\alpha$. Similarly one can define $\textnormal{w}^*$-slice of bounded subsets of $X^*$ determined by $x\in X$ and $\alpha>0$. A point $x\in C$ is said to be a denting point of $C$ if for every $\varepsilon>0$, there exists a slice $S(C, x^*, \alpha)$ of $C$ such that $x\in S(C, x^*, \alpha)$ and $\operatorname{diam}(S(C, x^*, \alpha))<\varepsilon$, that is, the family of all slices of $C$ containing $x$ is neighborhood base of $x$ with respect to the relative norm topology. Similarly, one can define a \textnormal{w}$^*$-denting point of a bounded subset of $X^*$. It is well known that a point $x^*\in S(X^*)$ is a \textnormal{w}$^*$-denting point of $X_{1}^*$ if and only if it is a \textnormal{w}$^*$- $\|\cdot\|$ PC and an extreme point of $S(X^*)$. (see \cite{LLB}). We recall the following characterizations of \textnormal{w}$^*$-denting points. 

\begin{thm}\cite[Corollary~2.12]{HL93}\label{T11}
    Suppose that $(\Omega, \Sigma, \mu)$ is non-atomic and $h\in L^{p}(\mu, X)^*$. Then $h$ is \textnormal{w}$^*$-denting point of $L^{p}(\mu, X)_{1}^*$ if and only if $h$ is a \textnormal{w}$^*$- $\|\cdot\|$ PC of $L^{p}(\mu, X)_{1}^*$ 
\end{thm}

\begin{thm}\cite[Theorem~3]{LLB}\label{T12}
    An element $h$ in $L^{p}(\mu, X)_{1}^*$ with $\|h\|=1$ is a \textnormal{w}$^*$-denting of $L^{p}(\mu, X)_{1}^*$ if and only if $\frac{h(t)}{\|h(t)\|}$ is a \textnormal{w}$^*$-denting point of $X_{1}^*$ for almost every $t\in \Omega$ with $h(t)\neq 0$.
\end{thm}
Therefore, a point $h$ in $L^{p}(\mu, X)_{1}^*$ is a \textnormal{w}$^*$-$\|\cdot\|$ PC of $L^{p}(\mu, X)_{1}^*$ if and only if, for almost every $t$ with $h(t)\neq 0$, $\frac{h(t)}{\|h(t)\|}$ is both a \textnormal{w}$^*$-$\|\cdot\|$ PC and an
extreme point of $X_{1}^*$. 

In this paper, we provide a direct proof of this result and give an example to show that the converse does not hold without the additional assumption that $\frac{h(t)}{\|h(t)\|}$ is an extreme point of $X_{1}^*$ for almost every $t$ with $h(t)\neq 0$. In the context of \textnormal{w}$^*$-\textnormal{w} PCs, for $1<p<\infty$, we show that a point $h$ in $L^{p}(\mu, X)_{1}^*$ is a \textnormal{w}$^*$-\textnormal{w} PC of $S\big(L^{p}(\mu, X)^*\big)$ if and only if, for almost every $t$ with $h(t)\neq 0$, $\frac{h(t)}{\|h(t)\|}$ is a is a \textnormal{w}$^*$-\textnormal{w} PC of $S(X^*)$ when $X^*$ is separable.

Now, we recall the definiton of smooth points and very smooth points.

\begin{defn}\cite{DJ}\label{D1}
    A point $x\in S(X)$ is said to be a smooth point if there exists a unique linear functional $x^*\in X_{1}^*$ such that $x^*(x)=1$.
\end{defn}
 Very smoothness is a notion intermediate between G\^ateaux and Fr\'echet differentiability (see, for example, \cite{DJ}). 
\begin{defn}\cite{DJ}\label{D2}
    A point $x\in S(X)$ is said to be a very smooth point of $X$ if it is a smooth point in $X$ and continues to be a smooth point in $X^{**}$ under the canonical embedding.
\end{defn}
The connection between smoothness and our topic comes from the equivalence of (i) and (iv) in \cite[Remark~3.1]{GI}. We recall this result below.

\begin{lem}\cite[Remark~3.1]{GI}\label{L15}
Let $x\in S(X)$ is a smooth, and let $x^*$ is the unique norming functional of $x$. Then the following statements are equivalent:
\end{lem}
\begin{enumerate}[label=(\roman*)]
    \item $x$ is very smooth of $X$.
    \item The duality mapping is n-\textnormal{w} upper semicontinuous at $x$.
    \item $x^*$ has a unique norm-preserving extension from $X$ to $X^{**}$.
    \item $x^*$ is a \textnormal{w}$^*$-\textnormal{w} PC of $S(X^*)$.
\end{enumerate}
In \cite[Theorem~1.1]{DK}, the authors characterize the smooth points of $L^{p}(\mu, X)$ in terms of smooth points of $X$, under the assumption that $X^*$ is separable. In this paper, we provide a complete characterization of very smooth points of $L^{p}(\mu, X)$ in terms of very smooth points of $X$, when $X^*$ is separable.

Let $(\Omega, \Sigma, \mu)$ be a measure space, where $\mu$ is a non-atomic probability measure. We use $L^p(\mu, X)$ to denote the space of all (equivalent classes of) $X$-valued \textbf{Bochner $p$-integrable} functions, equipped with the norm $\|f\|=\left(\int_{\Omega}\|f(\omega)\|^{p}d(\mu(\omega))\right)^{\frac{1}{p}}$ for $f\in L^p(\mu, X)$. We refer the reader to \cite{DU} for more information on this topic. Let $\mu$ be a finite measure and suppose that $X^*$ has the Radon-Nikodym property (RNP) with respect to $\mu$. By \cite[Theorem~IV.1.1]{DU}, we then have $L^{p}(\mu, X)^{*} = L^{q}(\mu, X^{*})$, where $1<p\leq q<\infty$ are such that $\frac{1}{p}+\frac{1}{q}=1$ and $L^1(\mu, X)^*=L^{\infty}(\mu, X^*)$, the space of all $X^*$-valued strongly measurable functions that are essentially bounded. More generally, if $X$ is any Banach space, we still have the canonical embedding $L^{q}(\mu, X^*)\subseteq L^p(\mu, X)^*$. Let $A \subseteq \Omega$ be a measurable subset of positive measure, and set $\Sigma' = \Sigma|_{A}$ (the restricted $\sigma$-algebra) and $\lambda = \mu|_{A}$.  
If $X$ has the RNP with respect to $\mu$, then it is easy to see that $X$ also has the RNP with respect to $\lambda$.

In Section~2, we obtain the following main results.

\begin{thm}[Section~2, Theorem~2.5]
    Let $1\leq p<\infty$ and let $X$ be a Banach space such that $X^*$ is separable, so that $L^{p}(\mu, X)^*=L^{q}(\mu, X^*)$. If $h\in L^{q}(\mu, X^*)_{1}$ is a \textnormal{w}$^*$-\textnormal{w} PC of $ S\big(L^{q}(\mu, X^*)\big)$, then $\frac{h(t)}{\|h(t)\|}$ is a \textnormal{w}$^*$-\textnormal{w} PC of $S(X^*)$ for almost every $t\in [0, 1]$ with $h(t)\neq 0$.
\end{thm}
The converse of the above theorem is not true for $p=1$ (see Example~\ref{E211}).
\begin{thm}[Section~2, Theorem~2.10]
    Let $X$ be a separable Banach space such that $X^*$ has the RNP, so that $L^{p}(\mu, X)^*=L^{q}(\mu, X^*)$. Let $h\in  S\big(L^{q}(\mu, X^*)\big)$ be such that $\frac{h(t)}{\|h(t)\|}$ is a \textnormal{w}$^*$-\textnormal{w} PC of $S(X^*)$ for almost every $t$ with $h(t)\neq 0$. Then $h$ is a \textnormal{w}$^*$-\textnormal{w} PC of $S\big(L^{q}(\mu, X^*)\big)$.
\end{thm}

In Section~3, we provide the following characterization which shows an application of weak$^*$ to weak continuity to study smooth points in higher order duals of $L^p(\mu, X)$ for $1<p<\infty$. For $p=1$, we give a counter example (see Remark~\ref{R34})

\begin{thm}[Section~3, Theorem~3.1]
    Let $X$ be a Banach space such that $X^*$ is separable and let $f\in S(L^{p}(\mu, X))$. If $f$ is a very smooth point of $L^{p}(\mu, X)$, then $\frac{f(t)}{\|f(t)\|}$ is a very smooth point of $X$ for almost every $t\in [0, 1]$ with $f(t)\neq 0$.
\end{thm}

    \section{\texorpdfstring{Weak$^*$-weak continuity in the dual unit ball of $L^{p}(\mu, X)$}{Continuity in the dual unit ball of Lp(mu, X)}}
 In this section, for $1< p< \infty$, we establish a characterization of \textnormal{w}$^*$-\textnormal{w} PCs of the dual unit ball of $L^p(\mu, X)_{1}^*$ when $X^*$ is separable. Let $\tau$ denote either the weak or the weak$^{*}$ topology on $X^{*}$. For $A \subseteq X^{*}$, we write $\overline{A}^{\tau}$ for the $\tau$-closure of $A$ in $X^{*}$ and $\overline{A}$ for the norm closure of $A$ in $X^*$. For a subset $B$ of a topological space $K$, we denote by $B^{o}$ the interior of $B$ in $K$. Let $X$ be a Banach space and let $K\subseteq X$ be a subset of $X$. We write $\operatorname{co}(K)$ to denote the convex hull of $K$. 
 
 The following lemma is an interesting application of the Hahn-Banach separation theorem. We include its proof for the sake of completeness.
\begin{lem}\label{L58}
    Let $(y_{k})\in X_{1}$ be a sequence and let $\varepsilon>0$. Then, there exists an element $x^*\in X_{1}^*$ such that $x^*(y_{k})\geq\varepsilon$ for each $k\in\mathbb{N}$ if and only if $\big\|\sum\lambda_{k}y_{k}\big\|\geq\varepsilon$ for all finitely supported sequences $\lambda=(\lambda_{k})_{k\ge1}\in \mathbb{Q}^{+}$ with $\sum_{k} \lambda_{k}=1$.
\end{lem}
\begin{proof}
    It is easy to see the forward implication. For the converse, let $C=\overline{\text{co}}\big(\{y_{k}:k\in\mathbb{N}\}\big)$. By the hypothesis, we have that $\varepsilon X_{1}^{o}\cap C=\emptyset$. Since $\varepsilon X_{1}^{o}$ is open convex and $C$ is convex, by applying the Hahn-Banach separation theorem (see \cite[Chapter~3, Theorem~3.4]{RW}), there exists an element $x^*\in X^*$ such that $x^*(u)<\alpha\leq x^*(y)$ for all $y\in C$ and $u\in \varepsilon X_{1}^{o}$. Taking supremum over all elements $u\in \varepsilon X_{1}^{o}$, we get that $\varepsilon\|x^*\|\leq \alpha$. This implies that $\frac{x^*}{\|x^*\|}(y_{k})\geq\varepsilon$ for each $k\in\mathbb{N}$. This completes the proof.
\end{proof}
\begin{lem}\label{L56}
    Let $X$ be a separable Banach space and let $U(X)$ be the set of all \textnormal{w}$^*$-\textnormal{w} PCs of $S(X^*)$. Then $U(X)$ is a $G_\delta$-set in $X_{1}^*$ in the \textnormal{w}$^*$-topology. 
\end{lem}
\begin{proof}
    Since $X$ is separable, we have that $K=(X_{1}^*, \textnormal{w}^*)$ is a compact metric space and hence, $K\times K^{\mathbb{N}}$ is a compact space. An element $x^*\in S(X^*)$ belongs to $U(X)$ if and only if there exists a sequence $(z_{k}^*)\in X_{1}^*$ such that $z_{k}^*$ converges to $x^*$ in the \textnormal{w}$^*$-topology, but the convergence does not hold in the w-topology. Let $\Lambda$ denote the set of all finitely supported sequences
$\lambda=(\lambda_{k})_{k\ge1}$ of non-negative rational numbers with
$\sum_{k}\lambda_{k}=1$; that is,
\[
  \Lambda=\Bigl\{(\lambda_{k})\in\mathbb{Q}^{\mathbb{N}}\ :\ \lambda_{k}\ge0\ \ \forall k,\quad
  \operatorname{supp}((\lambda_{k}))\ \text{is finite},\quad \sum_{k}\lambda_{k}=1\Bigr\}.
\]
Then $\Lambda$ is countable, and for any bounded sequence $(y_{k})_{k\ge1}$ in
$X^{*}$ the sum $\sum_{k}\lambda_{k}y_{k}$ is a finite convex combination. For each $\varepsilon\in \mathbb{Q}^{+}$, we define
\[
  C_{\varepsilon}=\left\{\bigl(x^{*},(z_{k})\bigr)\in K\times K^{\mathbb{N}}\ :\
  \begin{aligned}
    &d(z_{k},x^{*})\le\tfrac{1}{k}\quad\text{for all }k\ge1,\\
    &\Bigl\|\sum_{k}\lambda_{k}(z_{k}-x^{*})\Bigr\|\ge\varepsilon
      \quad\text{for all }(\lambda_{k})\in\Lambda
  \end{aligned}
  \right\}.
\]
It is easy to see that for each fixed $k\in\mathbb{N}$ and $(\lambda_{k})\in \Lambda$, the mappings $(x^{*},(z_{n}))\mapsto d(z_{k},x^{*})$ and $(x^{*},(z_{n}))\mapsto\big\|\sum_{n}\lambda_{n}(z_{n}-x^{*})\big\|$ defined on $K\times K^{\mathbb{N}}$ are continuous. Therefore, $C_{\varepsilon}$ is a closed subset of $K\times K^{\mathbb{N}}$, and hence it is a compact set. Let $\pi:K\times K^{\mathbb{N}}\to K$ be the projection onto the first coordinate. Then, $\pi(C_{\varepsilon})$ is compact in $K$. Hence, $\pi(C_{\varepsilon})$ is a \textnormal{w}$^*$-compact subset of $X_{1}^*$. By Lemma~\ref{L58}, it is easy to see that $X_{1}^*\setminus U(X)=\bigcup_{\varepsilon\in\mathbb{Q}^{+}}\pi(C_{\varepsilon})$. Hence, $U(X)=\bigcap_{\varepsilon\in\mathbb{Q}^{+}}\pi(C_{\varepsilon})^c$, where each $\pi(C_{\varepsilon})^c$ is a \textnormal{w}$^*$-open subset of $X_{1}^*$. This completes the proof.
\end{proof}

\begin{rem}
Using an argument similar to the one used in the proof of the above lemma, we get that set of all \textnormal{w}$^*$-$\|\cdot\|$ PCs of $S(X^*)$ forms a $G_{\delta}$-set in $X_{1}^*$ in the \textnormal{w}$^*$-topology.
\end{rem}

Let $Y$ be a closed subspace of $X$. In general, there is no connection between the \textnormal{w}$^*$-\textnormal{w} PCs of $S(Y^*)$  and of $S(X^*)$ (see \cite[Theorem~8.2]{SMR}). We recall from \cite[Chapter~1]{HWW} that $Y$ is said to be an $M$-ideal in $X$ if there exists a norm one projection $P:X^*\to X^*$ such that $P(X^*)=Y^{\perp}$ and $\|x^*\|=\|P(x^*)\|+\|x^*-P(x^*)\|$ for all $x^*\in X^*$. In this case, we have $X^*=Y^{\perp}\oplus_{1}Y^*$, where $Y^{\perp}=\{x^*\in X^*:x^*(y)=0 \text{ for all } y\in Y\}$. It is well known that if $Y$ is an $M$-ideal in $X$, then a point $y^*\in S(Y^*)$ that is a \textnormal{w}$^*$-\textnormal{w} PC of $S(Y^*)$, is a \textnormal{w}$^*$-\textnormal{w} PC of $S(X^*)$ (see \cite[Theorem~6.2]{SMR}). In the following lemma, we extend this result for $\ell^p$-summands. 

For $1\leq p<\infty$ and Banach spaces $Y$ and $Z$, $Y\oplus_{p}Z$ denotes the $\ell^{p}$-direct sum of $Y$ and $Z$, equipped with the norm $\|(y, z)\|_{p}=\big(\|y\|^{p}+\|z\|^{p}\big)^{1/p}$. By duality, we have $X^*=Y^{^*}\oplus_{q}Z^*$. For a point $y^*\in S(Y^*)$, we denote by $(y^*, 0)$ the point in $S(X^*)$ defined by $(y^*, 0)(y, z)=y^*(y)$ for all $(y, z)\in Y\oplus_{p}Z$.

\begin{lem}\label{L57}
    Let $Y\subseteq X$ be a closed subspace of $X$ such that $X=Y\oplus_{p}Z$. Then, a point $y^*\in Y_{1}^*$ is a \textnormal{w}$^*$-\textnormal{w} PC of $S(Y^*)$ if and only if $(y^*, 0)\in X_{1}^*$ is a \textnormal{w}$^*$-\textnormal{w} PC of $S(X^*)$.
\end{lem}
\begin{proof}
   
    Let $y^*\in S(Y^*)$ be a \textnormal{w}$^*$-\textnormal{w} PC of $S(Y^*)$. Let $(y_{\alpha}^*, z_{\alpha}^*)\in S(X^*)$ be a net converging to $(y^*, 0)$ in the \textnormal{w}$^*$-topology of $X^*$. Then, for each $y\in Y$, we have $y_{\alpha}^*(y)\to y^*(y)$. As $\|(y_{\alpha}^*, z_{\alpha}^*)\|_{q}=1$, we have $\|y_{\alpha}^*\|\leq 1$. Since $y^*$ is a \textnormal{w}$^*$-\textnormal{w} PC of $S(Y^*)$, we have that $y_{\alpha}^*\to y^*$ in the w-topology of $Y$. By the \textnormal{w}$^*$-lower semicontinuity of the norm, we have $1=\|y^*\|\leq \liminf_{\alpha}\|y_{\alpha}^*\|\leq 1$. Therefore, by passing on to a subnet if necessary, we may assume that $\|y_{\alpha}^*\|\to 1$. Consequently, we have $\|z_{\alpha}^*\|\to 0$ and hence $z_{\alpha}^*\to 0$ in the w-topology of $Z$. Let $(y^{**}, z^{**})\in Y^{**}\oplus_{p}Z^{**}$. Then, $(y^{**}, z^{**})(y_{\alpha}^*, z_{\alpha}^*)=y^{**}(y_{\alpha}^*)+z^{**}(z_{\alpha}^*)\to y^{**}(y^*)$. This shows that $(y_{\alpha}^*, z_{\alpha}^*)\to (y^*, 0)$ in the w-topology of $X^*$. Hence, $(y^*, 0)$ is a \textnormal{w}$^*$-\textnormal{w} PC of $S(X^*)$.
    
    Conversely, let $(y^*, 0)\in S(X^*)$ be a \textnormal{w}$^*$-\textnormal{w} PC of $S(X^*)$. Let $(y_{\alpha}^*)\in S(Y^*)$ be a net converging to $y^*$ in the \textnormal{w}$^*$-topology of $Y^*$. Then, for each $(y, z)\in Y\oplus_{p}Z$, we have $(y_{\alpha}^*, 0)(y, z)=y_{\alpha}^*(y)\to y^*(y)=(y^*, 0)(y, z)$. This shows that $(y_{\alpha}^*, 0)\to (y^*, 0)$ in the \textnormal{w}$^*$-topology of $X^*$. Since $(y^*, 0)$ is a \textnormal{w}$^*$-\textnormal{w} PC of $S(X^*)$, we have that $(y_{\alpha}^*, 0)\to (y^*, 0)$ in the w-topology of $X^*$ and hence $y_{\alpha}^*\to y^*$ in the w-topology of $Y^*$. Hence, $y^*$ is a \textnormal{w}$^*$-\textnormal{w} PC of $S(Y^*)$. This completes the proof.
\end{proof}

Throughout, $([0,1], \mathcal{A}, \mu)$ denotes the Lebesgue measure space, where $\mathcal{A}$ is the $\sigma$-field of Lebesgue measurable sets and $\mu$ is Lebesgue measure. As this measure space is complete, we regard elements of the associated $L^{p}$-spaces as functions defined at every point of $[0,1]$. Let $V$ be a Polish space, that is, a separable and completely metrizable topological space. We write $\mathcal{B}_{V}$ for the Borel $\sigma$-algebra of $V$ and $\mathcal{A} \otimes \mathcal{B}_{V}$ for the product $\sigma$-algebra on $[0,1] \times V$ (see \cite[pp.~83, 87]{S}).

Let $1 < p < \infty$, let $q$ be its conjugate exponent, and let $X$ be a Banach space whose dual $X^{*}$ has the RNP, so that $L^{p}(\mu, X)^{*} = L^{q}(\mu, X^{*})$. Fix $h \in S(L^{p}(\mu, X)^{*})$ and put $Z = \{t \in [0,1] : h(t) = 0\}$. Suppose that $\mu(Z) > 0$. Then $L^{p}(\mu, X) = L^{p}(Z, X) \oplus_{p} L^{p}(Z^{c}, X)$, and since $h$ vanishes on $Z$, we may regard $h$ as an element of $S(L^{p}(Z^{c}, X)^{*})$. By Lemma~\ref{L57}, $h$ is a \textnormal{w}$^{*}$-\textnormal{w} PC of $S(L^{p}(\mu, X)^{*})$ if and only if it is a \textnormal{w}$^{*}$-\textnormal{w} PC of $S(L^{p}(Z^{c}, X)^{*})$.

In the following theorem, we show that if $h$ is a \textnormal{w}$^*$-\textnormal{w} PC of $S(L^{p}(\mu, X)^*)$, then $\frac{h(t)}{\|h(t)\|}$ is a \textnormal{w}$^*$-\textnormal{w} PC of $S(X^*)$ for almost every $t\in [0, 1]$ with $h(t)\neq 0$. For $f\in L^{p}(\mu, X)$, we define $|f|:\Omega\to \mathbb{R}$ by $|f|(t)=\|f(t)\|$. Clearly, $|f|\in L^{p}(\mu)$.
\begin{thm}\label{T59}
    Let $X$ be a Banach space such that $X^*$ is separable. If $h\in L^{p}(\mu, X)_{1}^*$ is a \textnormal{w}$^*$-\textnormal{w} PC of $ S\big(L^{p}(\mu, X)^*\big)$, then $\frac{h(t)}{\|h(t)\|}$ is a \textnormal{w}$^*$-\textnormal{w} PC of $S(X^*)$ for almost every $t\in [0, 1]$ with $h(t)\neq 0$.
\end{thm}
\begin{proof} Since $X^*$ is separable, by \cite[Theorem~III.3.1]{DU}, the dual space $X^*$ has the RNP. Therefore, the dual space $L^p(\mu, X)^*$ can be identified with the space $L^q(\mu, X^*$). Let $h\in L^{q}(\mu, X^*)_{1}$ be a \textnormal{w}$^*$-\textnormal{w} PC. Let $\Lambda$, $C_{\varepsilon}$ (for $\varepsilon\in\mathbb{Q^{+}})$, $U(X)$ and $\pi$ be as in the proof of Lemma~\ref{L56}. We first establish the proof when $\mu(Z)=0$. Without loss of generality, we may assume that $h(t)\neq 0$ for every $t\in [0, 1]$. 
    
     Suppose $A=\{t\in [0, 1]: \frac{h(t)}{\|h(t)\|}\in S(X^*)\textnormal{ is not a w$^*$-w PC}\}$. Clearly, $A=\Big(\frac{h}{|h|}\Big)^{-1}\Big(\big(X_{1}^*\setminus U(X)\big)\cap S(X^*)\Big)$. We first show that $A$ is a measurable subset of $[0, 1]$. For $n\in\mathbb{N}$, let $U_{n}=\big\{x^*\in X_{1}^*:\|x^*\|>1-\frac{1}{n}\big\}$. Since the norm on $X^*$ is \textnormal{w}$^*$-lower semicontinuous, it gives that each $U_{n}$ is \textnormal{w}$^*$-open subset of $X_{1}^*$. Clearly, $S(X^*)=\bigcap_{n\geq 1}U_{n}$. Thus $S(X^*)$ is a  \textnormal{w}$^*$-Borel subset of $X_{1}^*$. It follows from Lemma~\ref{L56} that $X_{1}^*\setminus U(X)$ is a \textnormal{w}$^*$-Borel set in $X_{1}^*$. Therefore, $\big(X_{1}^*\setminus U(X)\big)\cap S(X^*)$ is a \textnormal{w}$^*$-Borel subset of $X_{1}^*$. Since $\frac{h}{|h|}$ is \textnormal{w}$^*$-measurable, we have that $A$ is a measurable subset of $[0, 1]$.

    We now show that $\mu(A)=0$. Suppose not. Then $\mu(A)>0$. We construct a sequence in $L^{q}(\mu, X^*)_{1}$ converging to $h$ in the \textnormal{w}$^{*}$-topology but not weakly, contradicting the assumption that $h$ is a \textnormal{w}$^*$-\textnormal{w} PC. We have the following cases. 
    
    \begin{case} $p>1$.
    By Lemma~\ref{L56}, we have that $A\subseteq\bigcup_{\varepsilon\in\mathbb{Q}^{+}} \left( \frac{h}{|h|}\right)^{-1}\bigl(\pi(C_{\varepsilon})\bigr)$. Since $\mu(A)>0$, there exists an $\varepsilon\in\mathbb{Q}^{+}$ such that $\mu\left(\left( \frac{h}{|h|}\right)^{-1}\bigl(\pi(C_{\varepsilon})\bigr)\right)>0$. Let $B=\left( \frac{h}{|h|}\right)^{-1}\bigl(\pi(C_{\varepsilon})\bigr)$. Then, by the definition of $C_{\varepsilon}$, for each $t\in B$, there exists a sequence $(z_{k}^*(t))_{k\in \mathbb{N}}\subseteq X_{1}^*$ such that $z_{k}^*(t)\to \frac{h(t)}{\|h(t)\|}$ in the \textnormal{w}$^*$-topology but not in the \textnormal{w}-topology, that is, $\bigl(\frac{h(t)}{\|h(t)\|}, (z_{k}^*(t))\bigr)\in C_{\varepsilon}$. For $t\in B$, let $A_{t}=\{(z_{k}^*)\in K^{\mathbb{N}}: \Bigl(\frac{h(t)}{\|h(t)\|}, (z_{k}^*)\Bigr)\in C_{\varepsilon} \}$. Clearly, $A_{t}$ is a non-empty subset of $K^{\mathbb{N}}$. It is easy to check that for each fixed $t\in B$, the mapping $(z_{k}^*)\mapsto \Bigl(\frac{h(t)}{\|h(t)\|}, (z_{k}^*)\Bigr)$ defined on $K^{\mathbb{N}}$ is measurable. It is shown in Lemma~\ref{L56} that $C_{\varepsilon}$ is a closed subset of $K\times K^{\mathbb{N}}$. Therefore, for each $t\in B$, $A_{t}$ is a Borel subset of $K^{\mathbb{N}}$.

    We define a set-valued mapping $\Gamma: B\to 2^{K^{\mathbb{N}}}$ by $\Gamma(t)=A_{t}$. Then Graph$(\Gamma)=\Big\{(t, (z_{k}^*))\in B\times K^{\mathbb{N}}: \Bigl(\frac{h(t)}{\|h(t)\|}, (z_{k}^*)\Bigr)\in C_{\varepsilon} \Big\}$. An argument similar to the one used above gives that Graph$(\Gamma)$ is a measurable subset of $B \times K^{\mathbb{N}}$, since the mapping $(t, (z_{k}^*))\mapsto \Bigl(\frac{h(t)}{\|h(t)\|}, (z_{k}^*)\Bigr)$ is measurable and $C_{\varepsilon}$ is closed in $K\times K^{\mathbb{N}}$. By applying the von Neumann's selection theorem (see \cite[Chapter~5, Corollary~5.5.8]{SR}) to the measurable set Graph$(\Gamma)\subseteq B\times K^{\mathbb{N}}$, we obtain a measurable selection $g: B\to K^{\mathbb{N}}$ such that $g(t)\in A_{t}$ for each $t\in B$. Therefore, for each $t \in B$, $g(t) = (z_{k}^{*}(t))$ is a sequence in $X_{1}^{*}$ satisfying $\Bigl( \frac{h(t)}{\|h(t)\|},\, (z_{k}^{*}(t)) \Bigr) \in C_{\varepsilon}.$ For each $k\in \mathbb{N}$, we define a function $h_{k}:[0, 1]\to X^*$ by
    \[h_{k}(t)=
    \begin{cases}
       \|h(t)\|z_{k}^*(t) , & t \in B , \\
h(t), & t \notin B.
    \end{cases}
    \]
    Since $g$ is measurable and, on $B$, each $h_{k}$ is $|h|\pi_{k}\circ g$, where $\pi_{k}$ denotes the projection onto the $k$-th coordinate, it follows that each $h_{k}$ is measurable when $X^{*}$ is equipped with the w$^{*}$-topology. By the hypothesis, $X^*$ is a separable Banach space, by Pettis measurability theorem (see \cite[Chapter~II, Section~1, Theorem~2]{DU}), we get that each $h_{k}$ is strongly measurable with $\|h_{k}(t)\|\leq \|h(t)\|$, and hence $h_{k}\in L^{q}(\mu, X^*)$ for each $k\in \mathbb{N}$. It is easy to see that $\|h_{k}\|_{q}\leq 1$. 
    
    We now show that $h_{k}\to h$ in the \textnormal{w}$^*$-topology of $L^{q}(\mu, X^*)$ but not weakly. Let $f\in L^{p}(\mu, X)$. Then we have 
    \[
    \int_{0}^{1} \bigr(h_{k}(t)-h(t))(f(t)\bigl)\,d\mu(t)=\int_{B}\|h(t)\|\Bigl(z_{k}^*(t)-\frac{h(t)}{\|h(t)\|}\Bigr)(f(t))\,d\mu(t).
    \]
    By the choice of $(z_{k}^*(t))$, we have that for each $t\in B$, $z_{k}^*(t)\to \frac{h(t)}{\|h(t)\|}$ in the \textnormal{w}$^*$-topology of $X^*$. Hence, $\|h(t)\|\Bigl(z_{k}^*(t)-\frac{h(t)}{\|h(t)\|}\Bigr)(f(t))\to 0$ as $k\to\infty$, and $\left|\|h(t)\|\Bigl(z_{k}^*(t)-\frac{h(t)}{\|h(t)\|}\Bigr)(f(t))\right|\leq 2\|h(t)\|\|f(t)\|$. By H\"older's inequality, we have that $|h||f|\in L^{1}(\mu)$. Therefore, it follows from Dominated convergence theorem that $\int_{0}^{1} (h_{k}(t)-h(t))(f(t))\,d\mu(t)\to 0$. Thus, $h_{k}\to h$ in the \textnormal{w}$^*$-topology of $L^{q}(\mu, X^*)$. However, by the choice of $z_{k}^*(t)$, we have that $\bigl\|\bigr(\sum_{k}\lambda_{k}z_{k}^*(t)\bigl)-\frac{h(t)}{\|h(t)\|}\bigr\|\geq \varepsilon$ for each $(\lambda_{k})\in \Lambda$ and $t\in B$. Consequently, we have
    
    \begin{align*}
        \Big\|\Bigl(\sum_{k}\lambda_{k}h_{k}\Bigr)-h\Big\|_{q}&=\int_{B}\|h(t)\|\Big\|\Bigl(\sum_{k}\lambda_{k}z_{k}^*(t)\Bigr)-\frac{h(t)}{\|h(t)\|}\Big\|\,d\mu\\
        &\geq \varepsilon \int_{B}\|h(t)\|\,d\mu(t)\\
        &\geq \varepsilon M,
    \end{align*}
    where $M=\int_{B}\|h(t)\|\,d\mu(t)$. 
    \end{case}
\begin{case}
    $p=1$. By using the same argument as in Case~1, we can show that $h_{k}\to h$ in the \textnormal{w}$^*$-topology of $L^{\infty}(\mu, X^*)$ and  
    \begin{align*}
        \Big\|\Bigl(\sum_{k}\lambda_{k}h_{k}\Bigr)-h\Big\|_{\infty}&=\operatorname{ess\,sup}_{t\in B}\|h(t)\|\Big\|\Bigl(\sum_{k}\lambda_{k}z_{k}^*(t)\Bigr)-\frac{h(t)}{\|h(t)\|}\Big\|\\
        &\geq \varepsilon \big( \operatorname{ess\,sup}_{t\in B}\|h(t)\|\big)\\
        &\geq \varepsilon M',
    \end{align*}
    where $M'=\frac{1}{\mu(B)}\int_{B}\|h(t)\|\,d\mu(t)$. 
\end{case}
In both cases, the constants $M$ and $M'$ are positive finite numbers, since $\|h(t)\|\neq 0$ for each $t\in B$, and $\mu(B)>0$. This implies that $h\notin \overline{\operatorname{co}}^{\|.\|}\{h_{k}: k\in \mathbb{N}\}=\overline{\operatorname{co}}^{\textnormal{w}}\{h_{k}: k\in \mathbb{N}\}$, and hence $h_{k}$ does not converges to $h$ weakly. This contradicts the assumption that $h$ is a \textnormal{w}$^*$-\textnormal{w} PC. Hence, $\mu(A)=0$.

Now, suppose $0<\mu(Z)<1$ and $h\in L^{p}(\mu, X)_{1}^*$ is a \textnormal{w}$^*$-\textnormal{w} PC. Then, by Lemma~\ref{L57}, it follows that $h\in L^{p}(Z^c, X)_{1}^*$. Therefore, from the first part of the proof, we have that $\frac{h(t)}{\|h(t)\|}$ is a \textnormal{w}$^*$-\textnormal{w} PC of $S(X^*)$ for almost every $t\in Z^c$. This completes the proof.
\end{proof}

    

It is worth mentioning that the converse of Theorem~\ref{T59} need not hold for $p=1$ (see Example~\ref{E211}). For $1<p<\infty$, however, the converse does hold; this is Theorem~\ref{T28}. We first prove several lemmas that will be used to establish the converse.

For the converse, we need a sharpened form of Jensen's inequality from probability theory, due to Liao and Berg \cite{LGA}, which we now recall.

\begin{thm}\cite[Theorem~1]{LGA}
Let $V$ be a one dimension random variable with mean $\Delta V$, $P(X\in (a, b))=1$ (here $P$ denote the probability measure) and variance $\operatorname{var}(V)$. Let $\varphi$ be a twice differentiable function on $(a, b)$. Then 
\begin{align*}
    0\leq E[\varphi(V)]-\varphi(E[V])\leq \frac{1}{2}\sup_{t\in (a, b)}\varphi''(t)\operatorname{var}(V),
\end{align*}
where $\varphi''(t)$ denotes the second order derivative of $\varphi$ at $t$, $E[\varphi(V)]$ denotes the expected value of $\varphi(V)$, and $\operatorname{var}(V)=E[(V-\Delta V)^2]$.
\end{thm}
In our setup, a random variable is simply a real-valued measurable function on a probability space, and its mean and variance are the corresponding integrals. See the following remark.
 \begin{rem}\label{R27}
    Let $f\in L^p(\mu)$. In probabilistic terms, $|f|$ is a random variable on $([0,1],\mathcal{A},\mu)$, $\Delta |f| = {E}[|f|]=\int_{[0,1]} |f(t)|\,d\mu$ is its expectation, and $\operatorname{var}(|f|)=\int_{[0, 1]}(|f|-\Delta |f|)^2\,d\mu$ is its variance. Applying the above theorem to $|f|$ and $\varphi(t)=t^p$, we have
 \begin{align}
    \Delta|f|^{p}-\big(\Delta |f|\big)^{p}\leq \frac{1}{2}\sup_{t\in (0, 1)}\varphi''(t)\operatorname{var}\big(|f|\big).
 \end{align}
\end{rem}
 For a measurable subset $A$ of $[0, 1]$, Holder's inequality gives that $\int_{A}\|f(t)\|\,d\mu\leq \mu(A)^{\frac{1}{q}}\big\|\chi_{A}|f|\big\|_{p}$, where $\chi_{A}$ is the characteristic function of $A$. In the next lemma, we give an upper bound for the difference between the two sides of this inequality. This lemma plays an important role in the proof of Theorem~\ref{T28}. Thorughout this section, we consider $1<p<\infty$ and we denote by $C_{p}$ a constant which depends only on $p$.

\begin{lem}\label{L26}
    Let $f\in L^p(\mu, X)$, $A\subseteq [0, 1]$ be measurable with $\mu(A)>0$, and
$T\geq 2M>0$ satisfying $T-M\leq\|f(t)\|\leq T+M$ for every $t\in A$. Then $$\mu(A)^{\frac{1}{q}}\big\|\chi_{A}|f|\big\|_{p}-\int_{A}\|f(t)\|\,d\mu\leq C_{p}\mu(A)\frac{M^2}{T},$$
where $C_{p}$ is a constant which depends on $p$.
\end{lem}
\begin{proof}
   Let $\Delta|f|$ denote the mean of the function $|f|$ over $A$, that is, $\Delta|f|=\frac{1}{\mu(A)}\int_{A}\|f(t)\|\,d\mu$. Then, it is equivalent to show that
   \begin{align}\label{EL1}
    \big(\Delta|f|^{p}\big)^{\frac{1}{p}}-\Delta |f|\leq \frac{C_{p}M^2}{T}.
   \end{align}
   Let $I=[T-M, T+M]$. We define $\varphi:I\to \mathbb{R}$ by $\varphi(t)=t^p$. Clearly, $\varphi$ is a twice differentiable function on the interval $I$. Let $\nu=\frac{1}{\mu(A)}\mu|_{A}$. Then, by applying Remark~\ref{R27} to the measure space $(A,\mathcal{A}|_{A}, \nu)$ with $\chi_{A}|f|$ and $\varphi$, we have
 \[
\Delta|f|^{p}-\big(\Delta |f|\big)^{p}
\leq \tfrac{1}{2}\,\sup_{t\in I}\varphi''(t)\,\operatorname{var}\big(\chi_{A}|f|\big),
\]
where $\operatorname{var}\big(\chi_{A}|f|\big)=\frac{1}{\mu(A)}\int_{A}\Big(|f|(t)-\Delta|f|\Big)^2\,d\mu$. We have $\varphi''(t)=p(p-1)t^{(p-2)}$. Let $K=\sup_{t\in I}\varphi''(t)$. Since $M\leq \frac{T}{2}$, the interval $I$ is contained in $\big[\frac{T}{2}, \frac{3T}{2}\big]$. Consequently, $K\leq p(p-1)\big(\tfrac{3T}{2}\big)^{p-2}$ when $p\geq 2$, and
$K\leq p(p-1)\big(\tfrac{T}{2}\big)^{p-2}$ when $1<p<2$. In either case $K\leq C_{p}T^{p-2}$, where $C_{p}$ depends only on $p$. The well-known Popoviciu inequality says that (see \cite{PT}) $$\operatorname{var}\big(\chi_{A}|f|\big)\leq \frac{1}{4}\big(T+M-(T-M)\big)^2=\frac{1}{4}(2M)^2=M^2.$$ Therefore, we get 
\begin{align*}
\Delta|f|^{p}&\leq \frac{K}{2}\operatorname{var}\big(\chi_{A}|f|\big)+\big(\Delta |f|\big)^{p}\nonumber\\
&\leq \frac{1}{2}C_{p}T^{p-2}M^2+\big(\Delta |f|\big)^{p}.
\end{align*}
Taking $p$-th root to both sides of the above inequality and by applying the Bernoulli inequality $(1+x)^r\leq 1+rx$ for $x>-1$ and $0<r\leq 1$ we get
\begin{align*}
\big(\Delta|f|^{p}\big)^{\frac{1}{p }}\nonumber
&\leq \left(\frac{1}{2}C_{p}T^{p-2}M^2+\big(\Delta |f|\big)^{p}\right)^{\frac{1}{p}}\\
&= \Delta |f|\left(1+\frac{C_{p}T^{p-2}M^2}{2\big(\Delta |f|\big)^p}\right)^{\frac{1}{p}}\\
&\leq \Delta |f|+\frac{C_{p}T^{p-2}M^2}{2p\big(\Delta |f|\big)^{p-1}}.
\end{align*}
Since $\Delta |f|\geq T-M\geq \frac{T}{2}$, we have $\big(\Delta |f|\big)^{p-1}\geq \left(\frac{T}{2}\right)^{p-1}$. Therefore, we get
\begin{align*}
\big(\Delta|f|^{p}\big)^{\frac{1}{p}}\leq \Delta |f|+\frac{C_{p}T^{p-2}M^2}{2p\left(\frac{T}{2}\right)^{p-1}}=\Delta |f|+\frac{C_{p}M^2}{T},
\end{align*}
which establishes~\ref{EL1}, and hence completes the proof.
\end{proof}

For a Banach space $X$, let $L^{p}_{\textnormal{w}^*}(\mu, X^*)$ denote the space of all \textnormal{w}$^*$-measurable function from $[0, 1] \to X^*$ such that $\int_{[0, 1]} \|f(t)\|^p\,d\mu<\infty$, equipped with norm $\|f\|_{p}=\left(\int_{[0, 1]} \|f(t)\|^p\,d\mu\right)^{\frac{1}{p}}$. It is well-known that the dual space $L^p(\mu, X)^*$ can be identified with the space $L^{q}_{\textnormal{w}^*}(\mu, X^*)$ (see \cite[Chapter~1, Theorem~1.5.4]{PM}).


Let $A$ be a measurable subset of $[0, 1]$ with $\mu(A)>0$. we have $L^{p}_{\textnormal{w}^*}(A, X^{**})$ is an $\ell^p$-summand in $L^{p}_{\textnormal{w}^{*}}(\mu, X^{**})$. Let $P_{A}$ denote the projection $f \mapsto \chi_{A}f$ from $L^{p}_{\textnormal{w}^{*}}(\mu, X^{**})$ onto the subspace $L^{p}_{\textnormal{w}^*}(A, X^{**})$. For $\Lambda \in L^{p}_{\textnormal{w}^{*}}(\mu, X^{**})^{*}$, we write $\Lambda P_{A} = \Lambda \circ P_{A}$; this functional can be identified with the restriction of $\Lambda$ to $L^{p}_{\textnormal{w}^{*}}(A, X^{**})$. We use this notation in the following lemma.


\begin{lem}\label{L27}
Let $X$ be a Banach
space such that $X^*$ has the RNP. Let
$h\in L^{q}(\mu, X^*)$ and let $\Lambda\in L^{p}(\mu, X)^{***}$ be a norm-preserving extension of $h$ from $L^{p}(\mu, X)$ to $L^{p}(\mu, X)^{**}$. Then, for every measurable subset $A$ of $[0, 1]$, $\Lambda P_{A}$ is a norm-preserving extension of $\chi_{A}h$ from $L^{p}(A, X)$ to $L^{p}(A, X)^{**}$.

\end{lem}
\begin{proof}
For a measurable subset $A$ of $[0, 1]$ with $\mu(A)>0$, we have $$L^{p}(\mu, X)^{**}=L^{p}_{\textnormal{w}^{*}}(\mu, X^{**})=L^{p}_{\textnormal{w}^{*}}(A, X^{**})\oplus_{p}L^{p}_{\textnormal{w}^{*}}(A^c, X^{**}),$$
by duality, we get 
$$L^{p}_{\textnormal{w}^{*}}(\mu, X^{**})^*=L^{p}_{\textnormal{w}^{*}}(A, X^{**})^*\oplus_{q}L^{p}_{\textnormal{w}^{*}}(A^c, X^{**})^*.$$
Therefore, $\Lambda=\Lambda P_{A}+\Lambda P_{A^c}$ and $\|\Lambda\|^q=\|\Lambda P_{A}\|^q+\|\Lambda P_{A^c}\|^q$, where $\Lambda P_{A}\in L^{p}_{\textnormal{w}^{*}}(A, X^{**})^*$ and $\Lambda P_{A^c}\in L^{p}_{\textnormal{w}^{*}}(A^c, X^{**})^*$. Since $\Lambda$ is a norm-preserving extension of $h$, we have $\|\Lambda\|=\|h\|_{q}$. By the definition of $P_{A}$, for $f\in L^p(A, X)$, we have $\Lambda P_{A}(f)=\Lambda(\chi_{A}f)=\int_{A}h(t)(f(t))\,d\mu$. This shows that $\Lambda P_{A}$ is an extension of $\chi_{A}h$ from $L^{p}(A, X)$ to $L^{p}(A, X)^{**}$. Consequently, we have $\|\Lambda P_{A}\|\geq \|\chi_{A}h\|_{q}$. A similar argument shows that $\|\Lambda P_{A^c}\|\geq \|\chi_{A^c}h\|_{q}$. Since $\|h\|_{q}^q=\|\Lambda\|^q=\|\Lambda P_{A}\|^q+\|\Lambda P_{A^c}\|^q$ and $\|h\|_{q}^q=\|\chi_{A}h\|_{q}^q+\|\chi_{A^c}h\|_{q}^q$, we have $\|\Lambda P_{A}\|=\|\chi_{A}h\|_{q}$, which shows that $\Lambda P_{A}$ is a norm-preserving extension of $\chi_{A}h$ from $L^{p}(A, X)$ to $L^{p}(A, X)^{**}$. This completes the proof.
\end{proof}
The following remark is an immediate consequence of Lemma~\ref{L27}; we will use it in the proof of Theorem~\ref{T28}.
\begin{rem}\label{R29}
    Let $A, X, h\in L^{q}(\mu, X^*)$ and $\Lambda\in L^{p}(\mu, X)^{***}$ be as in Lemma~\ref{L27}. We have $L^{p}(\mu, X)^{***}=L^{p}_{\textnormal{w}^*}(\mu, X^{**})^*$ . For $G\in L^{p}_{\textnormal{w}^*}(\mu, X^{**})$, by Lemma~\ref{L27}, we have $\big\|\Lambda|_{L^{p}_{\textnormal{w}^{*}}(A, X^{**})}\big\|=\|\Lambda P_{A}\|=\big\|\chi_{A}h\big\|_{q}$. Hence, 
    \begin{align*}
        \big|\Lambda(\chi_{A}G)\big|\leq \big\|\Lambda P_{A}\big\|\big\|\chi_{A}G\big\|_{p}=\big\|\chi_{A}h\big\|_{q}\big\|\chi_{A}G\big\|_{p}.
    \end{align*}
     
\end{rem}
Suppose $g:[0, 1]\to \mathbb{R}$ is positive a function. For $0<a<b<\infty$, we define the set $$S_{(a, b, g)}=\{t\in [0, 1]: 0<g(t)<a \textnormal{ or } g(t)>b \}.$$ We now prove the following lemma.
\begin{lem}\label{L28}
    Let $g:[0, 1]\to\mathbb{R}$ be a positive measurable function defined almost everywhere on
$[0, 1]$, and let $f_{1},\dots,f_{k}\in L^{1}(\mu)$. Then, for every $\varepsilon>0$ there
exists $0<a<b<\infty$ such that $\int_{S_{(a, b, g)}}|f_{i}(t)|\, d\mu<\varepsilon$ for
$i=1,\dots,k$.
\end{lem}
\begin{proof}
    Let $(a_{n})$ be a decreasing sequence converging to zero and let $(b_{n})$ be an increasing sequence diverging to $\infty$. For each $n\in \mathbb{N}$, we denote the set $S_{(a_{n}, b_{n}, g)}$ by $S_{n}$. It is easy to check that $(S_{n})$ is a decreasing sequence of measurable sets. Therefore, by \cite[Chapter~1, Theorem~1.19]{WR}, we have that $\mu(S_{n})\to \mu(\cap_{n\geq 1} S_{n})$. 

    We claim that $\mu(\cap_{n\geq 1} S_{n})=0$. Let $t\in \cap_{n\geq 1} S_{n}$. We have $g(t)>0$, and hence there exists $k\in \mathbb{N}$ such that $a_{n}<g(t)$ for each $n\geq k$. Therefore, we have that $g(t)>b_{n}$ for each $n\geq k$, and hence $g(t)=\infty$. This implies $\cap_{n\geq 1} S_{n}\subseteq \{t\in[0, 1]: g(t)=\infty\}$. Since $g$ is defined almost everywhere on $[0, 1]$, it follows that $\mu(\cap_{n\geq 1} S_{n})=0$. Thus $\mu(S_{n})\to 0$. 
    
    Let $1\leq i\leq k$ be fixed. Then, it is easy to see that $\chi_{S_{n}}\big|f_{i}\big|\to 0$  pointwise almost everywhere as each point not in $\cap_{n\geq 1} S_{n}$ leaves the sets $S_{n}$ eventually. Since $\chi_{S_{n}}\big|f_{i}\big|\leq \big|f_{i}\big|\in L^1(\mu)$, by Dominated convergence theorem, we get that $\int_{S_{n}}|f_{i}(t)|\, d\mu\to 0$. Therefore, for each $\varepsilon>0$, there exists $n_{i}\in \mathbb{N}$ such that $\int_{S_{n_{i}}}|f_{i}(t)|\, d\mu<\varepsilon$. Suppose $m=\operatorname{max}\{n_{i}:1\leq i\leq k\}$. Since the sets $S_{n}$ are decreasing, $S_{m}\subseteq S_{n_{i}}$ for each $i$, and hence
$\int_{S_{m}}|f_{i}(t)|\, d\mu\leq\int_{S_{n_{i}}}|f_{i}(t)|\, d\mu<\varepsilon$ for
$1\leq i\leq k$. The conclusion therefore follows on taking $a=a_{m}$ and $b=b_{m}$.
\end{proof}
We recall from \cite[Chapter~III, Section~2, Lemma 2.14]{HWW} that a point $x^*\in S(X^*)$ is a \textnormal{w}$^*$-\textnormal{w} PC if and only if $x^*$ has a unique norm preserving extension from $X$ to $X^{**}$. Let $x^*\in X_{1}^*$ and let $x^{**}\in X^{**}$. By the Hahn-Banach theorem (see \cite[Chapter~3, Section~6]{JBC}), the set of
norm-preserving extensions of $x^*$ to $X^{**}$ is non-empty, and
$$
\sup_{\Lambda\in X^{***}}\Big\{\operatorname{Re}\big(\Lambda(x^{**})\big):\Lambda\in X_{1}^{***},\ 
\Lambda\big|_{X}=x^*\Big\}
=\inf_{x\in X}\Big\{\|x^{**}-x\|+\big(\operatorname{Re}x^{*}(x)\big)\Big\}.
$$
In particular, if $x^*$ is a \textnormal{w}$^*$-\textnormal{w} PC of $S(X^*)$, by
\cite[Chapter~III, Section~2, Lemma 2.14]{HWW}, the canonical image of $x^*$ in $X^{***}$ is its only norm-preserving
extension, so the supremum above equals $\operatorname{Re}x^{**}(x^*)$. Hence
$$
\inf_{x\in X}\Big\{\|x^{**}-x\|+\operatorname{Re}\big(x^{*}(x)\big)-\operatorname{Re}\big(x^{**}(x^*)\big)\Big\}=0 .
$$
Combining this fact with the lemmas established above, we now prove the converse of Theorem~\ref{T59} for $1 < p < \infty$.
\begin{thm}\label{T28}
    Let $X$ be a separable Banach space such that $X^*$ has the RNP, so that $L^{p}(\mu, X)^*=L^{q}(\mu, X^*)$. Let $h\in  S\big(L^{q}(\mu, X^*)\big)$ be such that $\frac{h(t)}{\|h(t)\|}$ is a \textnormal{w}$^*$-\textnormal{w} PC of $S(X^*)$ for almost every $t$ with $h(t)\neq 0$. Then $h$ is a \textnormal{w}$^*$-\textnormal{w} PC of $S\big(L^{q}(\mu, X^*)\big)$.
\end{thm}
\begin{proof}
    As $X^*$ has the RNP, $L^p(\mu, X)^*=L^q(\mu, X^*)$. Applying \cite[Chapter~1, Theorem~1.5.4]{PM} to $L^q(\mu, X^*)$, we obtain $L^p(\mu, X)^{**}=L_{\textnormal{w}^*}^p(\mu, X^{**})$. Let $Z=\big\{t\in [0, 1]:h(t)=0\big\}$. We first prove the theorem in the case when $\mu(Z)=0$, that is, when $h(t)\neq 0$ almost everywhere; the general case is deduced at the end of the proof. Without loss of generality, we may assume that $h(t)\neq 0$ for every $t\in [0, 1]$. Suppose $h\in  S\big(L^{q}(\mu, X^*)\big)$ is such that $\frac{h(t)}{\|h(t)\|}$ is a \textnormal{w}$^*$-\textnormal{w} PC of $S(X^*)$ for almost every $t$. Redefining $h$ on a null set, with values a fixed \textnormal{w}$^*$-\textnormal{w} PC of
$S(X^*)$, we may assume that $\frac{h(t)}{\|h(t)\|}$ is a \textnormal{w}$^*$-\textnormal{w}
PC of $S(X^*)$ for every $t\in[0, 1]$. By the virtue of \cite[Chapter~III, Section~2, Lemma 2.14]{HWW}, we establish the proof by exhibiting a unique norm-preserving extension of $h$ from $L^p(\mu, X)$ to $L^p(\mu, X)^{**}$. Let $\hat{h}$ denote the canonical embedding of $h$ into $L^p(\mu, X)^{***}$, and let $\Lambda$ be an element of $L^p(\mu, X)^{***}$ such that $\Lambda\big|_{L^p(\mu, X)}=h$ and $\big\|\Lambda\big\|=\big\|h\big\|_{q}$. We show that $\Lambda$ is equal to $\hat{h}$, that is $$\Lambda(G)=\int_{[0, 1]}G(t)\big(h(t)\big)\quad \textnormal{for every }G\in L_{\textnormal{w}^*}^p(\mu, X^{**}).$$ The proof proceeds in the following steps. 
    \begin{step}\label{S1}
    In this step we construct, for each $\eta>0$, a $\textnormal{w}^*$-measurable function $D_{1}$ from $[0, 1]$ to $X^{**}$ such
that
\[
\big\|D_{1}(t)\big\|<\operatorname{Re}\Big(D_{1}(t)\big(\tfrac{h(t)}{\|h(t)\|}\big)\Big)+\eta
\qquad\text{for every } t\in[0, 1] .
\] 
    
    Let $G\in L_{\textnormal{w}^*}^p(\mu, X^{**})$. Since $\frac{h(t)}{\|h(t)\|}$ is a \textnormal{w}$^*$-\textnormal{w} PC of $X_{1}^*$ for every $t$, it follows that 
    \begin{align}\label{E1}
        \inf_{x\in X}\left\{\big\|G(t)-x\big\|+\operatorname{Re}\left(\frac{h(t)}{\big\|h(t)\big\|}(x)-G(t)\Big(\frac{h(t)}{\big\|h(t)\big\|}\Big)\right)\right\}=0\qquad \textnormal{for each }t\in [0, 1].
    \end{align}
    For each $x\in X$, we define $\phi_{x}:[0, 1]\to \mathbb{R}$ by $$\phi_{x}(t)=\big\|G(t)-x\big\|+\operatorname{Re}\left(\frac{h(t)}{\big\|h(t)\big\|}(x)-G(t)\Big(\frac{h(t)}{\big\|h(t)\big\|}\Big)\right).$$
Clearly, $\phi_{x}(t)\geq 0$ for each $t\in [0, 1]$. Since $X$ is separable and $X^*$ has the RNP, it follows from \cite[Section~2, Corollary~2.15]{RRP} that $X^*$ is separable. Let $\{x^*_n:n\in\mathbb{N}\}$ be a dense subset of $S(X^*)$. Since $G$ is \textnormal{w}$^*$-measurable, we get that for each $x^*_{n}$, the function $t\mapsto G(t)(x^*_{n})$ is measurable. Hence, the map $t\mapsto \big\|G(t)\big\|=\sup_{n\in\mathbb{N}}\{\big|G(t)(x^*_{n})\big|\}$ is measurable and therefore $t\mapsto \big\|G(t)-x\big\|$ is measurable. Now, it is easy to see that $\phi_{x}$ is a measurable function. 

Let $\big\{x_{n}:n\in\mathbb{N}\big\}$ be a dense subset of $X$. By~\eqref{E1}, for every $t\in[0, 1]$ and every $\eta>0$ the set $\big\{n:\phi_{x_{n}}(t)<\eta\big\}$ is non-empty. Hence, it has a least element $n(t)$, by the well-ordering principle. We define $g_{1}:[0, 1]\to X$ by $g_{1}(t)=x_{n(t)}$. We claim that $g_{1}$ is strongly measurable. For $m\in\mathbb{N}$, let $E_{m}=\big\{t:n(t)=m\big\}$. By the definition of $n(t)$, we have $$E_{m}=\big\{t:\phi_{x_{m}}(t)<\eta\big\}\cap \big(\cap_{j<m}\big\{t:\phi_{x_{j}}(t)\geq\eta\big\}\big).$$
Since each $\phi_{x}$ is measurable, it follows that the set $E_{m}$ is measurable for each $m\in \mathbb{N}$. By definition, $E_{m}$'s forms a disjoint collection of sets such that $[0, 1]=\bigcup_{m=1}^{\infty}E_{m}$. Therefore, it is easy to check that $\sum_{i=1}^{m}\chi_{E_{m}}x_m$ converges pointwise almost everywhere to $g_{1}$. Hence, $g_{1}$ is strongly measurable. Let $D_{1}=G-g_{1}$. Clearly, $D_{1}$ is a \textnormal{w}$^*$-measurable function from $[0, 1]$ to $X^{**}$. For each $t\in[0, 1]$, we have 
\begin{align*}
    \big\|D_{1}(t)\big\|&=\big\|G(t)-g_{1}(t)\big\|\\
    &=\big\|G(t)-x_{n(t)}\big\|\\
    &\leq \operatorname{Re}\left(\big(G(t)-x_{n(t)}\big)\Big(\frac{h(t)}{\big\|h(t)\big\|}\Big)\right)+\eta\\
    &=\operatorname{Re}\left(\big(G(t)-g_{1}(t)\big)\Big(\frac{h(t)}{\big\|h(t)\big\|}\Big)\right)+\eta\\
    &=\operatorname{Re}\left(D_{1}(t)\Big(\tfrac{h(t)}{\|h(t)\|}\Big)\right)+\eta.
\end{align*}
Similarly for every $\delta>0$, one can define a countably-valued and hence a strongly measurable function $g_{2}:[0, 1]\to S(X)$ such that $\operatorname{Re}\left(\frac{h(t)}{\big\|h(t)\big\|}(g_{2}(t))\right)>1-\delta$. Since $\big\|g_{2}(t)\big\|=1$ for each $t$, it follows that $g_{2}\in L^p(\mu, X)$.
    \end{step}
    \begin{step}
        In this step, we combine the functions $D_{1}$ and $g_{2}$ constructed in Step~\ref{S1} to produce a function $D$ from $[0, 1]\to X^{**}$ such that ``$G-D$'' is a strongly measurable function on $[0, 1]$ with values in $X$, and we partition the set $[0, 1]$ in such a way that in each piece the integrals $\int \big\|G(t)\big\|^p\,d\mu$,  $\int \big\|G(t)\big\|\big\|h(t)\big\|\,d\mu$ are sufficiently small. 

       We first describe the partition of $[0,1]$ that will be used throughout the proof. Let $\varepsilon>0$. Since $|h|$ is defined everywhere applying Lemma~\ref{L28} to functions $|G|^p$ and $|G||h|$, there exists $0<a<b<\infty$ such that 
       \begin{align}\label{EL211}
        \int_{S_{(a, b, |h|)}}\big\|G(t)\big\|^p\,d\mu< \Big(\frac{\varepsilon}{8}\Big)^p\quad\textnormal{and}\quad \int_{S_{(a, b, |h|)}}\big\|G(t)\big\|\big\|h(t)\big\|\,d\mu< \Big(\frac{\varepsilon}{8}\Big)^p<\frac{\varepsilon}{8}.
       \end{align}
       
    Suppose $A=S_{(a, b, |h|)}=\big\{t\in[0, 1]:0<\big\|h(t)\big\|<a \textnormal{ or }b<\big\|h(t)\big\|\big\}$. Then $A^c=\big\{t\in[0, 1]: a\leq\big\|h(t)\big\|\leq b\big\}$. By using an argument similar to the one used in the proof of Lemma~\ref{L28}, we can choose $M>0$ such that, on the set $B=\big\{t\in[0, 1]:a\leq\big\|h(t)\big\|\leq b, \big\|G(t)\big\|>M\textnormal{ or }\big\|D_{1}(t)\big\|>M \big\}$ the integrals $\int \big\|G(t)\big\|^p\,d\mu$,  $\int \big\|G(t)\big\|\big\|h(t)\big\|\,d\mu$ satisfy~\ref{EL211}. 
    Let $$B_{0}=\big\{t\in [0, 1]: a\leq\big\|h(t)\big\|\leq b, \big\|G(t)\big\|\leq M, \big\|D_{1}(t)\big\|\leq M \big\}.$$ Then, we have $B_{0}^c=A\cup B$ and
    \begin{align}\label{E2}
     \big|\Lambda\big(\chi_{B_{0}^c}G\big)-\int_{{B_{0}^c}}G\big(t\big)\big(h(t)\big)\,d\mu\big|&\leq \big|\Lambda\big(\chi_{B_{0}^c}G\big)\big|+\Big|\int_{{B_{0}^c}}G\big(t\big)\big(h(t)\big)\,d\mu\Big|\\
     &\leq \big\|\chi_{B_{0}^c}h\big\|_{q} \big\|\chi_{B_{0}^c}G\big\|_{p}+\int_{{B_{0}^c}}\big\|G(t)\big\|\big\|h(t)\big\| \nonumber\\
     &\leq  \big\|\chi_{B_{0}^c}G\big\|_{p}+\int_{{B_{0}^c}}\big\|G(t)\big\|\big\|h(t)\big\| \nonumber\\
     &\leq  \big\|\chi_{A}G\big\|_{p}+\big\|\chi_{B}G\big\|_{p}+\int_{{B_{0}^c}}\big\|G(t)\big\|\big\|h(t)\big\| \nonumber\\
     &\leq\frac{\varepsilon}{2}. \nonumber
    \end{align}
    Let $T\geq 2M$. Define a function $D:[0, 1]\to X^{**}$ by $D(t)=D_{1}(t)+Tg_{2}(t)$. We have $G-D=G-D_{1}-Tg_{2}=g_{1}-Tg_{2}$. Since $g_{1}$ and $g_{2}$ are both $X$-valued strongly measurable functions, it gives that $G-D$ satisfies the same. In the next step, for each $t\in [0, 1]$, we identify, the point $G(t)-D(t)$ with  its canonical image into $X^{**}$.
    
    \end{step}

    \begin{step}\label{S3}
        In this step we split $B_{0}$ into finitely many sets $C_{1},\dots,C_{k}$ such that, for each $i$, $\chi_{C_{i}}(G-D)\in L^p(\mu, X)$ and the following conditions hold:
\begin{enumerate}[label=\textup{(\roman*)}, leftmargin=*, itemsep=2pt]
 \item $\displaystyle
        \Lambda\big(\chi_{C_{i}}(G-D)\big)
        =\int_{C_{i}}\big(G(t)-D(t)\big)\big(h(t)\big)\,d\mu$
  \item $T-M\leq\big\|D(t)\big\|\leq T+M$ for every $t\in C_{i}$
  \item $\displaystyle
        \big\|D(t)\big\|\leq
        \operatorname{Re}\Big(D(t)\Big(\tfrac{h(t)}{\|h(t)\|}\Big)\Big)+\eta+T\delta$
        for every $t\in C_{i}$.
\end{enumerate}
Let $a$ and $b$ be as in Step~\ref{S1}. As $a<b$, we may choose $\delta'>0$ with $a(1+\delta')<b$. Suppose that $k$ is the smallest positive integer with
$a(1+\delta')^{k}\geq b$. For $1\leq i\leq k$ put
\[
B_{i}=\big[a(1+\delta')^{i-1},\, a(1+\delta')^{i}\big),
\qquad
C_{i}=\big\{t\in B_{0}:\|h(t)\|\in B_{i}\big\}.
\]
The sets $C_{1},\dots,C_{k}$ are pairwise disjoint with union $B_{0}$. By avoiding the null sets, we may assume that $\mu(C_{i})>0$ for each $i$. Since $G-D$ is strongly measurable and by definition, for each $t\in C_i$, $\big\|(G-D)(t)\big\|\leq 2M+T$, we get that $\chi_{C_{i}}(G-D)$ belongs to $L^p(\mu, X)$. Conditions
\textup{(i)} and \textup{(ii)} are readily verified for each of them. To verify the third condition, let $t\in C_{i}$. Then we have
\begin{align}
    \big\|D(t)\big\|&=\big\|D_{1}(t)+Tg_{2}(t)\big\|\\
    &\leq \big\|D_{1}(t)\big\|+T\nonumber.
\end{align}
From Step~\ref{S1}, we have that $\big\|D_{1}(t)\big\|\leq \operatorname{Re}\Big(D_{1}(t)\Big(\frac{h(t)}{\|h(t)\|}\Big)\Big)+\eta$. Thus, we obtain
\begin{align}\label{E:D}
\big\|D(t)\big\|
\leq\big\|D_{1}(t)\big\|+T
 &\leq\operatorname{Re}\Big(D_{1}(t)\Big(\frac{h(t)}{\|h(t)\|}\Big)\Big)+\eta+T\\
&=\operatorname{Re}\Big(D(t)\Big(\frac{h(t)}{\|h(t)\|}\Big)\Big)
  +T\Big(1-\operatorname{Re}\Big(\frac{h(t)}{\|h(t)\|}\big(g_{2}(t)\big)\Big)\Big)+\eta \nonumber \\
&\leq\operatorname{Re}\Big(D(t)\Big(\frac{h(t)}{\|h(t)\|}\Big)\Big)+T\delta+\eta ,\nonumber 
\end{align}
where the last inequality uses
$1-\operatorname{Re}\Big(\frac{h(t)}{\|h(t)\|}\big(g_{2}(t)\big)\Big)\leq\delta$ from
Step~\ref{S1}.
    \end{step}
    \begin{step}
        Finally, in this step we prove that $$\operatorname{Re}\big(\Lambda(G)\big)-\operatorname{Re}\Big(\int_{[0, 1]}G(t)\big(h(t)\big)\Big)<\varepsilon,$$
         for suitable choices of $\eta, \delta, \delta'$ and $T$. Since $G-D$ belongs to the space $L^p(\mu, X)$, we have
        $$\operatorname{Re}\big(\Lambda(\chi_{C_{i}}G-\chi_{C_{i}}D)\big)=\operatorname{Re}\left(\int_{C_{i}}(G(t)-D(t))(h(t))\,d\mu\right).$$ 
        Hence,
\begin{align}\label{E6}
   \operatorname{Re}\big(\Lambda(\chi_{C_{i}}G)\big)-\operatorname{Re}\left(\int_{C_{i}}G(t)(h(t))\,d\mu\right)=\operatorname{Re}\big(\Lambda(\chi_{C_{i}}D)\big)-\operatorname{Re}\left(\int_{C_{i}}\big\|(h(t))\big\|D(t)\Big(\frac{h(t)}{\|h(t)\|}\Big)\,d\mu\right)
\end{align}
Fix $1\leq i\leq k$ and put $c_{i}=a(1+\delta')^{i-1}$. By Remark~\ref{R29}, $\operatorname{Re}\big(\Lambda(\chi_{C_{i}}D)\big)
\leq\big\|\chi_{C_{i}}h\big\|_{q}\,\big\|\chi_{C_{i}}D\big\|_{p},$ and by
Step~\ref{S3}, for every $t\in C_{i}$,
\[
\big\|D(t)\big\|-\eta-T\delta
\leq\operatorname{Re}\Big(D(t)\Big(\tfrac{h(t)}{\|h(t)\|}\Big)\Big),
\qquad
c_{i}\leq\big\|h(t)\big\|\leq c_{i}(1+\delta').
\] 
Putting these values in equation~\ref{E6}, we get that
\begin{align}
    \operatorname{Re}\big(\Lambda(\chi_{C_{i}}G)\big)-\operatorname{Re}\left(\int_{C_{i}}G(t)(h(t))\,d\mu\right)&\leq c_{i}(1+\delta')\mu(C_{i})^{\frac{1}{q}}\big\|\chi_{C_{i}}D\big\|_{p}-c_{i}\int_{C_{i}}\big(\big\|D(t)\big\|-\eta-T\delta\big)\,d\mu.
\end{align}
For $1\leq i\leq k$, applying Lemma~\ref{L26} to the function $\chi_{C_{i}}D\in L^{p}(\mu, X^{**})$ and to the set $C_{i}$, we get 
        \begin{align*}
            \mu(C_{i})^{\frac{1}{q}}\big\|\chi_{C_{i}}|D|\big\|_{p}-\int_{C_{i}}\|D(t)\|\,d\mu\leq C_{p}\mu(C_{i})\frac{M^2}{T},
        \end{align*}
        where $C_{p}$ is a constant which depends on $p$. Therefore, we have
        \begin{align*}
            \operatorname{Re}\big(\Lambda(\chi_{C_{i}}G)\big)-\operatorname{Re}\left(\int_{C_{i}}G(t)(h(t))\,d\mu\right)&\leq c_{i}(1+\delta')\left(\int_{C_{i}}\|D(t)\|\,d\mu+C_{p}\mu(C_{i})\frac{M^2}{T}\right)\\
            &-c_{i}\int_{C_{i}}\big\|D(t)\big\|\,d\mu+c_{i}\mu(C_{i})(\eta+T\delta)\\
            &\leq c_{i}\mu(C_{i})\delta'(T+M)+c_{i}(1+\delta')C_{p}\mu(C_{i})\frac{M^2}{T}+c_{i}\mu(C_{i})(\eta+T\delta).
        \end{align*}
        We can choose $\delta'>0$ small enough such that $(1+\delta')<2$. This gives that
        \begin{align}\label{}
            \operatorname{Re}\big(\Lambda(\chi_{C_{i}}G)\big)-\operatorname{Re}\left(\int_{C_{i}}G(t)(h(t))\,d\mu\right)\leq c_{i}\mu(C_{i})\Big(\delta'(T+M)+2\big(C_{p}\frac{M^2}{T}+\eta+T\delta\big)\Big),    
        \end{align}
        for $1\leq i\leq k$. Since $B_{0}=\bigcup_{i=1}^{k}C_{i}$, we have
        \begin{align*}
            \operatorname{Re}\big(\Lambda(\chi_{B_{0}}G)\big)-\operatorname{Re}\left(\int_{B_{0}}G(t)(h(t))\,d\mu\right)&=\sum_{i=1}^{k}\left(\operatorname{Re}\big(\Lambda(\chi_{C_{i}}G)\big)-\operatorname{Re}\left(\int_{C_{i}}G(t)(h(t))\,d\mu\right)\right)\\
            &\leq \sum_{i=1}^{k}c_{i}\mu(C_{i})\Big(\delta'(T+M)+2\big(C_{p}\frac{M^2}{T}+\eta+T\delta\big)\Big).
        \end{align*}
        
       Since $c_{i}\leq \big\|h(t)\big\|$ for each $t\in C_{i}$, we have
     $$\sum_{i=1}^{k}c_{i}\mu(C_{i})\leq \int_{B_{0}}\|h(t)\|\,d\mu\leq \mu(B_{0})^{\frac{1}{p}}.$$ Also, we have $a\mu(B_{0})\leq \int_{B_{0}}\|h(t)\|\,d\mu\leq \mu(B_{0})^{\frac{1}{p}}$, and hence $\mu(B_{0})\leq a^{-q}$. Let $K=a^{-\frac{q}{p}}$. Then, we have 
       \begin{align*}
        \operatorname{Re}\big(\Lambda(\chi_{B_{0}}G)\big)-\operatorname{Re}\left(\int_{B_{0}}G(t)(h(t))\,d\mu\right)&\leq K\Big(\delta'(T+M)+2\big(C_{p}\frac{M^2}{T}+\eta+T\delta\big)\Big).
       \end{align*}
        Since the choices of $\delta',\eta$, $\delta$ and $T$ are independent, by choosing them $(\delta',\eta,\delta)$ small enough and $T$ large enough, we can make the right hand side of the above inequality less than $\frac{\varepsilon}{2}$. Combining this with~\eqref{E2}, we get 
        \begin{align*}
            \operatorname{Re}\big(\Lambda(G)\big)-\operatorname{Re}\int_{[0, 1]}G(t)(h(t))\,d\mu&<\varepsilon,
        \end{align*}
        Since $\varepsilon>0$ is arbitrary, we have
        $$\operatorname{Re}\big(\Lambda(G)\big)\leq\operatorname{Re}\int_{[0, 1]}G(t)(h(t))\,d\mu.$$
        Applying the same argument to $-G$ and $iG$, we get that 
        $$
        \Lambda(G)=\int_{[0, 1]}G(t)(h(t))\,d\mu.$$
    \end{step} 
Thus, $\hat{h}$ is the unique norm-preserving extension of $h$ from $L^{p}(\mu, X)$ to $L^{p}(\mu, X)^{**}$, and hence $h$ is a \textnormal{w}$^*$-\textnormal{w} PC of $S(L^{q}(\mu, X^*))$.
\end{proof}

We now give an example which shows that the converse of Theorem~\ref{T28} fails for $p=1$. Recall that a Banach space $X$ is said to be Hahn-Banach smooth if every element of the unit sphere $S(X^*)$ admits a unique norm-preserving extension from $X$ to $X^{**}$. Equivalently, every element of the unit sphere $S(X^*)$ is a \textnormal{w}$^*$-\textnormal{w} PC. If $L^{1}(\mu)$ is one dimensional, then $L^{1}(\mu, X)=X$, and therefore $L^{1}(\mu, X)$ is Hahn-Banach smooth if and only if $X$ is. If $\textnormal{dim}\big(L^{1}(\mu)\big)\geq 2$, then it is easy to see that $L^{1}(\mu, X)$ is Hahn-Banach smooth if and only if $L^{1}(\mu)$ is finite dimensional and $X$ is reflexive, that is, if and only if $L^{1}(\mu, X)$ is reflexive. 

\begin{ex}\label{E211}
    Let $X$ be a reflexive Banach space and let $L^1(\mu)$ is infinite dimensional. Since $L^{1}(\mu, X)$ is not Hahn-Banach smooth, there exists an element $h\in S(L^1(\mu, X)^*)$ which is not a \textnormal{w}$^*$-\textnormal{w} PC. Clearly, each unit vector $\frac{h(t)}{\|h(t)\|}$ with $h(t)\neq 0$ is a \textnormal{w}$^*$-\textnormal{w} PC of $(X^*)$, since $X$ is a reflexive Banach space.
\end{ex}
We now discuss the relationship between the \textnormal{w}$^*$- $\|\cdot\|$ PCs of $L^{p}(\mu,X)_{1}^*$ and those of $X_{1}^*$. It follows from Theorems~\ref{T11} and~\ref{T12} that when $h$ is a \textnormal{w}$^*$-$\|\cdot\|$ PC of $L^{p}(\mu,X)_{1}^*$, then $\frac{h(t)}{\|h(t)\|}$ is a \textnormal{w}$^*$-$\|\cdot\|$ PC of $X_{1}^*$ for almost every $t$ with $h(t)\neq 0$. We reach this conclusion only indirectly, as the argument goes through the characterization of \textnormal{w}$^*$-denting points. A direct proof can be given by following the proof of Theorem~\ref{T59}, with the set $C_{\varepsilon}$ replaced by the set
\[
    C'_{\varepsilon}
    = \Bigl\{ \bigl(x^{*}, (z_{k})_{k \ge 1}\bigr) \in S(X^{*}) \times K^{\mathbb{N}} :
      d(z_{k}, x^{*}) \le \tfrac{1}{k} \ \text{and} \
      \|z_{k} - x^{*}\| \ge \varepsilon \ \text{for all } k \ge 1 \Bigr\}.
\]
Since the argument is otherwise the same, we omit the details.


\begin{thm}\label{T513}
    Let $X$ be such that $X^*$ is separable. If $h\in L^{p}(\mu, X)_{1}^*$ is a \textnormal{w}$^*$- $\|\cdot\|$ PC, then $\frac{h(t)}{\|h(t)\|}$ is a \textnormal{w}$^*$- $\|\cdot\|$ PC in $S(X^*)$ for almost every $t$ with $h(t)\neq 0$.
\end{thm}

Example~\ref{E511} shows that the converse of Theorem~\ref{T513} fails in general. With an additional assumption, it becomes true, see the following remark.

\begin{rem}
    As mentioned in the introduction, if a point $x^*\in X_{1}^*$ is a \textnormal{w}$^*$- $\|\cdot\|$ PC and an extreme point of $X_{1}^*$, then it is a \textnormal{w}$^*$-denting point. Let $h\in L^{q}(\mu, X^*)$ be such that $\frac{h(t)}{\|h(t)\|}$ is both a \textnormal{w}$^*$-$\|\cdot\|$ PC and an extreme point of $X_{1}^*$ for almost every $t$ with $h(t)\neq 0$. Then $\frac{h(t)}{\|h(t)\|}$ is a \textnormal{w}$^*$-denting point of $X_{1}^*$ for almost every $t$. Hence, by Theorem~\ref{T11} and Theorem~\ref{T12}, $h$ is a \textnormal{w}$^*$- $\|\cdot\|$ PC of $L^{p}(\mu, X)_{1}^*$.
\end{rem}

 The following example shows that the extremality assumption on image points $\frac{h(t)}{\|h(t)\|}$ cannot be dropped. We recall the Rademacher functions $r_{n} : [0,1] \to \mathbb{R}$
(see \cite{JL}), defined by $r_{n}(t) = \operatorname{sign}\bigl(\sin(2^{n}\pi t)\bigr)$ for $n \in \mathbb{N}.$
It is well known that
\[
    \int_{0}^{1} r_{n}(t)\,\psi(t) \, d\mu \longrightarrow 0
    \quad \text{as } n \to \infty, \ \text{for every } \psi \in L^{1}(\mu)
\]
(see \cite[Proposition~1.d.1]{JL}), that is, $r_{n} \to 0$ in the
$\textnormal{w}^{*}$-topology of $L^{\infty}(\mu)$. We use this fact in the following
example.

\begin{ex}\label{E511}
    Let $X=c_{0}$, so that $X^{*}=\ell^{1}$ is a separable Banach space. Let $\mu$ be the Lebesgue measure on
$[0,1]$. Then
$L^{1}(\mu,c_{0})^{*}=L^{\infty}(\mu,\ell^{1})$. Put
\[  h(t)\equiv\tfrac{1}{2}(e_{1}+e_{2}),\qquad
  v=\tfrac{1}{2}(e_{1}-e_{2}),
\]
so that $\|h(t)\|_{1}=\|v\|_{1}=1$ for every $t$ and hence $\|h\|_{\infty}=1$. Since weak$^*$ and norm topologies coincide on $S(\ell^{1})$, the function $h(t)$ is a \textnormal{w}$^*$-\textnormal{w} PC of $S(\ell^{1})$ for every $t\in[0,1]$. Let $(h_{n})$ be a sequence defined by
\[
  h_{n}(t)=h+r_{n}(t)\,v=
  \begin{cases}
    e_{1}, & r_{n}(t)=1,\\
    e_{2}, & r_{n}(t)=-1,
  \end{cases}
\]
satisfies $\|h_{n}(t)\|_{1}=1$ for every $t$, whence $h_{n}\in S(L^{\infty}(\mu,\ell^{1}))$. Suppose $f\in L^{1}(\mu,c_{0})$ and define $\psi:[0,1]\to\mathbb{C}$ by
$\psi(t)=\langle v,f(t)\rangle$. Then $|\psi(t)|\le\|v\|_{1}\|f(t)\|_{\infty}$,
so $\psi\in L^{1}(\mu)$, and
\[
  \langle h_{n}-h,f\rangle=\int_{0}^{1}r_{n}(t)\,\psi(t)\,d\mu\longrightarrow0 .
\]
Hence $h_{n}\to h$ in the \textnormal{w}$^{*}$-topology of $L^{\infty}(\mu,\ell^{1})$. Since $\|h_{n}(t)-h(t)\|=\|v\|_{1}=1$ for all $t\in [0, 1]$, we get that $\|h_{n}-h\|_{q}=1\not\to 0 $.
\end{ex}


 

\section{Very smooth points of $L^{p}(\mu, X)$}
In this section, we apply the results on \textnormal{w}$^*$-\textnormal{w} PCs obtained in the previous section to study very smooth points of $L^{p}(\mu, X)$. We provide a characterization of very smooth points of $L^{p}(\mu, X)$ for $1<p<\infty$ in terms of the very smooth points of $X$ under the assumption that $X^*$ is separable.

For a point $x\in S(X)$, we denote the associated state space by $S_{x}=\{x^*\in S(X^*): x^*(x)=1\}$ and for $y\in S(Y)$, we denote the associated state space by $S_{y}^{Y^*}=\{y^*\in S(Y^*): y^*(y)=1\}$. Recall from Definitions~\ref{D1} and~\ref{D2} that a point $x \in S(X)$ is called a smooth point if $S_{x}$ is a singleton, and a very smooth point if, in addition, $x$ remains a smooth point of $X^{**}$ under the canonical embedding. Let $Y\subseteq X$ be a closed subspace of $X$. In \cite[Lemma~2.7]{SD1}, the authors have shown that if $Y$ is $L$-summand in $X$, then no point of $S(Y)$ can be a smooth point of $S(X)$. 

In the next lemma, we recall from \cite{DK} the stability of smooth points in $\ell^{p}$-summand subspaces of $X$ for $1<p<\infty$. We include a proof for the sake of sake completeness.

\begin{lem}\label{L515}
    Let $Y\subseteq X$ be a closed subspace of $X$ such that $X=Y\oplus_{p}Z$ for some closed subspace $Z$ of $X$. Let $x\in S(X)$ be such that $x=(y, z)$, where $y\in Y$, $z\in Z$. Then the following statements are true.
    \bla 
    \item If $y$ and $z$ are both non-zero, then $x$ is a smooth point of $X$ if and only if $\frac{y}{\|y\|}$ is a smooth point of $Y$ and $\frac{z}{\|z\|}$ is a smooth point of $Z$.
    
    \item If $y=0$, then $x$ is a smooth point of $X$ if and only if $z$ is a smooth point of $Z$.
    \el 
\end{lem}
\begin{proof} $(a)$ In this case, we prove that $$S_{x}=\|y\|^{p-1}S_{\frac{y}{\|y\|}}^{Y^*} \times \|z\|^{p-1}S_{\frac{z}{\|z\|}}^{Z^*}.$$ Suppose $x^*\in S_{x}$. We write $x^*=(y^*, z^*)$ for some $y^*\in Y^*$ and $z^*\in Z^*$. Then, we have $y^*(y)+z^*(z)=1$. Since $\|x\|=1$, we have that $\|y\|^{p}+\|z\|^{p}=1$. By the H\"{o}lder's inequality, we get 
\begin{align}\label{E7}
    1=y^*(y)+z^*(z)&\leq \|y^*\|\|y\|+\|z^*\|\|z\| \nonumber\\
    &\leq (\|y^*\|^{q}+\|z^*\|^{q})^{1/q}\cdot (\|y\|^{p}+\|z\|^{p})^{1/p}=1.
\end{align}
    This implies that equality holds in the H\"{o}lder's inequality. So the coordinates $(\|y^*\|^q, \|z^*\|^q)$ and $(\|y\|^p, \|z\|^p)$ are proportional, that is, $(\|y^*\|^q, \|z^*\|^q)=\lambda(\|y\|^p, \|z\|^p)$ for some $\lambda>0$. Since $\|x^*\|=\|x\|=1$, it gives that $\lambda=1$. Therefore, we have $\|y^*\|=\|y\|^{p-1}$ and $\|z^*\|=\|z\|^{p-1}$. It follows from ~\ref{E7} that $y^*(y)=\|y^*\|\|y\|$ and $z^*(z)=\|z^*\|\|z\|$. Thus, the points $\frac{y}{\|y\|}\in S(Y)$ and $\frac{z}{\|z\|}\in S(Z)$ attains its norm at $\frac{y^*}{\|y\|^{p-1}}\in S(Y^*)$ in $Y$ and $\frac{z^*}{\|z\|^{p-1}}\in S(Z^*)$ respectively. Hence we have that $S_{x}\subseteq \|y\|^{p-1}S_{\frac{y}{\|y\|}}^{Y^*} \times \|z\|^{p-1}S_{\frac{z}{\|z\|}}^{Z^*}$. The converse inclusion can be proved similarly. Thus, $x$ is a smooth point of $X$ if and only if $\frac{y}{\|y\|}$ is a smooth point of $Y$ and $\frac{z}{\|z\|}$ is a smooth point of $Z$.

$(b)$ If $z=0$, then we have that $x=y\in S(Y)$. In this case, it is easy to see that $S_{x}=S_{y}^{Y^*} \times \{0\}$. Thus, $x$ is a smooth point of $X$ if and only if $y$ is a smooth point of $Y$.
\end{proof}
The following remark, an immediate consequence of Lemma~\ref{L515}, gives an
analogous characterization of the very smooth points of $X$ when $X$ admits an
$\ell_{p}$-decomposition for some $1<p<\infty$.
\begin{rem}\label{R516}
Let $1<p<\infty$ and let $Y\subseteq X$ be a closed subspace of $X$ such that $X=Y\oplus_{p}Z$ for some closed subspace $Z$ of $X$. Let $x\in S(X)$ be a smooth point such that $x=(y, z)$ for some $y\in Y$ and $z\in Z$. Suppose that $y$ and $z$ both are non-zero. By duality, we have $X^{**}=Y^{**}\oplus_{p}Z^{**}$. Therefore, by Lemma~\ref{L515}, the point $x$ smooth is in $X^{**}$ if and only if $\frac{y}{\|y\|}$ is smooth in $Y^{**}$ and $\frac{z}{\|z\|}$ is smooth in $Z^{**}$. Hence $x$ is a very smooth point of $X$ if and only if $\frac{y}{\|y\|}$ is a very smooth point of $Y$ and $\frac{z}{\|z\|}$ is a very smooth point of $Z$. A similar argument can be applied when either $y$ or $z$ is zero.  
\end{rem}

We use the relation of very smooth points and \textnormal{w}$^*$-\textnormal{w} PCs together with Lemma~\ref{L515} and Remark~\ref{R516} to establish the following theorem. 
\begin{thm}
    Let $1\leq p<\infty$, $X$ be a Banach space such that $X^*$ is separable and let $f\in S(L^{p}(\mu, X))$. If $f$ is a very smooth point of $L^{p}(\mu, X)$, then $\frac{f(t)}{\|f(t)\|}$ is a very smooth point of $X$ for almost every $t\in [0, 1]$ with $f(t)\neq 0$. The converse does hold for $1<p<\infty$.
\end{thm}
\begin{proof}
    Let $X$ be such that $X^*$ is separable and let $f\in S(L^{p}(\mu, X))$ be a very smooth point. Suppose $Z=\{t\in [0, 1]: f(t)=0\}$.
    
    \smallskip
\begin{case} $p=1$. 
   We first observe that $\mu(Z)=0$. Indeed, let $\mu(Z)>0$. Then we have $$L^{1}(\mu, X)=L^1(Z, X)\oplus_{1}L^1(Z^c, X).$$ We may consider $f$ as an element of $L^1(Z^c, X)$. Then by \cite[Lemma~2.7]{SD1}, we get that $f$ is not a smooth point of $L^{1}(\mu, X)$ and hence, it not a very smooth point. This is a contradiction. Therefore, $\mu(Z)=0$. Since $f$ is a very smooth point, there exists a unique $g\in S(L^{1}(\mu, X)^*)$ such that $g(f)=\int_{0}^{1}g(t)(f(t))\, dt = 1$ and $g$ is a \textnormal{w}$^*$-\textnormal{w} PC of $S(L^{1}(\mu, X)^*)$ (see Lemma~\ref{L15}). It follows from \cite[Theorem~2.3]{SD1} that $\frac{f(t)}{\|f(t)\|}$ is a smooth point of $X$ for almost every $t\in [0, 1]\setminus Z$. Since $\|f(t)\|\neq 0$ almost everywhere, by \cite[Remark~3.2]{SD4}, we have that $g(t)\neq 0$ whenever $f(t)\neq 0$ and $\frac{g(t)}{\|g(t)\|}$ is a unique linear functional in $S(X^*)$ such that $\frac{g(t)}{\|g(t)\|}\bigl(\frac{f(t)}{\|f(t)\|}\bigr) = 1$ for almost every $t\in [0, 1]\setminus Z$. Since $X^*$ is separable and $g$ is a \textnormal{w}$^*$-\textnormal{w} PC of $S(L^{1}(\mu, X)^*)$, by Theorem~\ref{T59} we get that $\frac{g(t)}{\|g(t)\|}$ is a \textnormal{w}$^*$-\textnormal{w} PC of $S(X^*)$ for almost every $t\in [0, 1]\setminus Z$. Hence $\frac{f(t)}{\|f(t)\|}$ is a very smooth point of $S(X)$ for almost every $t\in [0, 1]\setminus Z$.
\end{case}  

\smallskip
\begin{case} $1<p<\infty$. We first note that all the results that are used in case~1, are also true for $1<p<\infty$. If $\mu(Z)=0$, then an argument similar to the one use in the case $p=1$ gives that $\frac{f(t)}{\|f(t)\|}$ is a very smooth point of $S(X)$ for almost every $t\in [0, 1]\setminus Z$. If $\mu(Z)>0$, then we have $L^{p}(\mu, X)=L^{p}(Z^c, X)\oplus_{p} L^{p}(Z, X)$. We consider $f$ as an element of $L^{p}(Z^c, X)$. Since $f$ is a very smooth point of $L^{p}(\mu, X)$, it follows from Remark~\ref{R516}, that $f$ is a very smooth point of $L^{p}(Z^c, X)$. Therefore, we get that $\frac{f(t)}{\|f(t)\|}$ is a very smooth point of $X$ for almost every $t\in [0, 1]\setminus Z$. This completes the proof of the necessary part.

Let $1<p<\infty$ and let $\frac{f(t)}{\|f(t)\|}$ is a very smooth point of $X$ for almost every $t\in [0, 1]\setminus Z$. Then, by \cite[Theorem~1.1]{DK}, we have that $f$ is a smooth point of $L^{p}(\mu, X)$. Let $g\in  S\big(L^{p}(\mu, X)^*\big)$ be the unique norming functional of $f$. As in case~1, we get that $\frac{g(t)}{\|g(t)\|}$ is a unique norming functional of $\frac{f(t)}{\|f(t)\|}$. Since $\frac{f(t)}{\|f(t)\|}$ is a very smooth point of $X$ for almost every $t\in [0, 1]\setminus Z$, it follows from Lemma~\ref{L15} that $\frac{g(t)}{\|g(t)\|}$ is a \textnormal{w}$^*$-\textnormal{w} PC of $S(X^*)$ for almost every $t\in [0, 1]\setminus Z$. It follows from Theorem~\ref{T59} that $g$ is a \textnormal{w}$^*$-\textnormal{w} PC of $ S\big(L^{p}(\mu, X)^*\big)$. Hence, $f$ is a very smooth point of $L^{p}(\mu, X)$. This completes the proof.
\end{case}
\end{proof}

\begin{rem}\label{R34}
The converse of the above theorem is not true for $p=1$ (see, \cite[Example~3.16]{SD4}).
\end{rem}
For $1\leq p<\infty$ and a Banach space $X$, it is well known that the space $L^{q}(\mu, X^*)$ is isometrically isomorphic to a subspace of $L^{p}(\mu, X)^*$. An argument similar to the one use in the proof of \cite[Lemma~3.11]{SD4} shows that the dual unit ball $L^{q}(\mu, X^*)_{1}$ is \textnormal{w}$^*$-dense in $L^{p}(\mu, X)_{1}^*$. Therefore, \textnormal{w}$^*$-\textnormal{w} PCs in $L^{p}(\mu, X)_{1}^*$ belong to $L^{q}(\mu, X^*)_{1}$. This observation leads to the following theorem.

\begin{thm}
    If each unit vector in $L^{p}(\mu, X)$ attains its norm some \textnormal{w}$^*$-\textnormal{w} PC in $L^{p}(\mu, X)_{1}^*$, then $X^*$ has the RNP.
\end{thm}
\begin{proof}
    Suppose that each unit vector in $L^{p}(\mu, X)$ attains its norm some \textnormal{w}$^*$-\textnormal{w} PC in $L^{p}(\mu, X)_{1}^*$. Since \textnormal{w}$^*$-\textnormal{w} PCs in $L^{p}(\mu, X)_{1}^*$ belong to $L^{q}(\mu, X^*)_{1}$, we have that each unit vector in $L^{p}(\mu, X)$ attains its norm at some point in $L^{q}(\mu, X^*)_{1}$. Hence, the conclusion follows from \cite[Corollary~1.2]{JR}.
\end{proof}

\subsection{Open Problems: } We finish this paper with a list of a few open problems concerning \textnormal{w}$^*$-\textnormal{w} PCs and smooth points in the setting of Bochner function spaces.

\begin{qn}
Characterize the \textnormal{w}$^*$-\textnormal{w} PCs of $L^{p}(\mu, X)_{1}^*$ in terms of the \textnormal{w}$^*$-\textnormal{w} PCs of $X_{1}^*$ for $1<p<\infty$ remains open for those Banach spaces $X$ whose duals are not separable?
\end{qn}

\begin{qn}
 Characterize very smooth points of $L^{p}(\mu, X)$ in terms of very smooth points of of $X$ for $1<p<\infty$ remains open for those Banach spaces $X$ whose duals are not separable?
\end{qn}

\section*{Acknowledgements}
The author would like to thank Prof.~T.~S.~S.~R.~K.~Rao and Dr.~Priyanka Grover for their invaluable guidance, constant support, and inspiring mentorship throughout this work. Sincere thanks are due to the Department of Mathematics, Shiv Nadar Institution of Eminence, Delhi-NCR, India for providing a stimulating academic environment that greatly facilitated this research.

\bibliographystyle{plain, abbrv}

\end{document}